\documentclass{amsart}
\input epsf

\usepackage{amsfonts,amsthm,amsmath,amssymb,latexsym}
\usepackage{graphicx,color}
\usepackage[all]{xy}
\usepackage{labelfig}
\usepackage{epsfig}

\usepackage{amsfonts,amsthm,amsmath,amssymb,latexsym}

\theoremstyle{definition}

 \newtheorem{definition}{Definition}[section]
 \newtheorem{remark}[definition]{Remark}

\theoremstyle{plain}

 \newtheorem{proposition}[definition]{Proposition}
 \newtheorem{theorem}[definition]{Theorem}
 
 \newtheorem{lemma}[definition]{Lemma}

\author{Hideki Miyachi and Dragomir \v Sari\' c}
\address[Hideki Miyachi]{Division of Mathematical and Physical Sciences,
Graduate School of Natural Science \& Technology,
Kanazawa University,
Kakuma-machi, Kanazawa,
Ishikawa, 920-1192, Japan}
\email{miyachi@se.kanazawa-u.ac.jp}
\address[Dragomir {\v S}ari{\' c}]{Department of Mathematics Queens College of CUNY 65-30 Kissena Blvd. Flushing, NY 11367
USA
and
Mathematics PhD Program The CUNY Graduate Center 365 Fifth Avenue
New York, NY 10016-4309 USA}
\email{Dragomir.Saric@qc.cuny.edu}

\thanks{The first author is partially supported by JSPS KAKENHI Grant Numbers
16K05202,
16H03933,
17H02843.
The second author was partially supported by a Simons Foundation grant,
Grant Number 346391.}

\begin{document}

\date{\today}
\title[Convergence of Teichm\"uller deformations]{Convergence of Teichm\"uller deformations in the universal Teichm\"uller space}

\maketitle

\begin{abstract}
Let $\varphi :\mathbb{D}\to\mathbb{C}$ be an integrable holomorphic function on the unit disk 
$\mathbb{D}$ and  $D_{\varphi}:\mathbb{D}\to T(\mathbb{D})$  the corresponding Teichm\"uller 
disk in the universal Teichm\"uller space $T(\mathbb{D})$. For a positive $t$ it is known that 
$D_{\varphi}(t)\to [\mu_{\varphi}]\in PML_b(\mathbb{D})$ as $t\to 1$, where
$\mu_{\varphi}$ is a bounded measured lamination representing a point on the Thurston boundary of 
$T(\mathbb{D})$. 
We extend this result by showing that 
$D_{\varphi}\colon \mathbb{D}\to T(\mathbb{D})$
extends as a continuous map from the closed disk $\overline{\mathbb{D}}$
to the Thurston bordification. In addition, we prove that the rate of convergence of 
$D_{\varphi}(\lambda )$  when $\lambda\to e^{i\theta}$ is independent of the type of the
approach to $e^{i\theta}\in\partial\mathbb{D}$.
\end{abstract}

\section{Introduction}

Let $\mathbb{D}$ be the unit disk equipped with the hyperbolic metric and $f:\mathbb{D}\to\mathbb{D}$ a quasiconformal map. Then $f$ extends by continuity to a
quasisymmetric map $h:S^1\to S^1$ of the unit circle $S^1$. Conversely, a quasisymmetric map $h:S^1\to S^1$ extends to a quasiconformal map of $\mathbb{D}$ and there are 
infinitely many such extensions (See \cite{GL}, \cite{LV}). 

The {\it universal Teichm\"uller space} $T(\mathbb{D})$ consists of all quasiconformal maps $f:\mathbb{D}\to\mathbb{D}$ up to an equivalence relation (See \cite{L}, \cite{GL}). 
Namely, two quasiconformal maps $f_1,f_2:\mathbb{D}\to\mathbb{D}$ are equivalent if there exists a conformal map $c:\mathbb{D}\to\mathbb{D}$ such that $f_2^{-1}\circ c\circ 
f_1$ extends by continuity to the identity on $S^1$. We will use an equivalent definition (See \cite{L}, \cite{GL}):
$$
T(\mathbb{D})=\{ h:S^1\to S^1:\ h\ \mathrm{ is\ quasisymmetric\ and\ fixes }\ 1,\ i,\ -1\}. 
$$
The Thurston boundary of the universal Teichm\"uller space $T(\mathbb{D})$ is identified with the space $PML_b(\mathbb{D})$ of projective bounded measured laminations on
$\mathbb{D}$
(See \S\ref{sec:Thurston-bordification}. See also \cite{BonSar},  \cite{Sar1}). 
In this paper, we describe the closure of Teichm\"uller disks in the Thurston bordification $T(\mathbb{D})\cup PML_b(\mathbb{D})$ of the universal Teichm\"uller space $T(\mathbb{D})$. In particular, if a sequence in the parameter of the Teichm\"uller disk converges to a point on the unit circle the corresponding sequence in $T(\mathbb{D})$ converges to a unique point in $PML_b(\mathbb{D})$ independently of the type of approach (e.g. along a geodesic, along a horocycle or even outside any horoball).

In previous works \cite{HaSar1}, \cite{HaSar2}, \cite{HaSar3}, Hakobyan and the second author of this paper showed that for an integrable holomorphic quadratic differential $\varphi$ on $\mathbb{D}$, the corresponding Teichm\"uller geodesic of $T(\mathbb{D})$ has a unique limit point on Thurston boundary of $T(\mathbb{D})$. The limit point $[\mu_{\varphi}]\in PML_b(\mathbb{D})$ is the projective class of the transverse measure to the geodesic straightening  of the vertical foliation of $\varphi$
multiplied by the reciprocal of the length of the vertical leaves (See Remark \ref{remark:1}).
For an integrable holomorphic quadratic differential $\varphi$ on $\mathbb{D}$,
the {\it Teichm\"uller disk} is a holomorphic disk in $T(\mathbb{D})$
given by the family of Beltrami differentials
$\{\lambda \overline{\varphi}/|\varphi|\}_{\lambda\in \mathbb{D}}$.
Let $D_\varphi(\lambda)$ be the Teichm\"uller  equivalence class 
associated to the Beltrami differential $\lambda \overline{\varphi}/|\varphi|$.
 We will prove the following,
 which is a generalization of a result in \cite{HaSar1}.

\begin{theorem}[Teichm\"uller deformation has the limit]
\label{thm:Teich-has-limit}
As $\lambda\to e^{i\theta}\in \partial \mathbb{D}=S^1$, $D_\varphi(\lambda)$ converges to the projective class of $\mu_{e^{-i\theta}\varphi}$ in the Thurston boundary $PML_b(\mathbb{D})$ of $T(\mathbb{D})$. Furthermore, the Teichm\"uller disk $D_\varphi\colon \mathbb{D}\to T(\mathbb{D})$ is extended as a homeomorphism from the closed unit disk $\overline{\mathbb{D}}$ onto the image in the Thurston bordification $T(\mathbb{D})\cup PML_b(\mathbb{D})$.
\end{theorem}

\begin{remark}[Limiting measured laminations] \label{remark:1}
Almost all vertical leaves of $e^{-i\theta}\varphi$ have exactly two limit points on $S^1$ (See \cite{Strebel}). The measured lamination $\mu_{e^{-i\theta}\varphi}$ has support on the geodesic lamination of $\mathbb{D}$ obtained by replacing the (vertical) leaves of $e^{-i\theta}\varphi$ with geodesics (for the hyperbolic metric) of $\mathbb{D}$ that have the same endpoints as vertical leaves. The transverse measure of $\mu_{e^{-i\theta}\varphi}$ is given by the integral 
$$
\int_I \frac{1}{l(z)}dx
$$ 
where $l(z)$ is the length of the vertical trajectory through $z=x+yi\in I$ and $dx$ is the differential of the horizontal displacement  in the natural parameter $z=x+yi$ of $e^{-i\theta}\varphi$.  
Notice that (unlike in the case of closed surfaces) it is not simply the horizontal transverse measure induced by $\varphi$. In fact, almost all vertical leaves of $e^{-i\theta}\varphi$ have finite length (in the metric $|\sqrt{e^{-i\theta}\varphi (\zeta )}d\zeta|$)
and the transverse measure is scaled by the reciprocal of the length of the vertical leaves thus giving us a new type of limit when compared to closed surfaces (See \cite{HaSar1}). 
\end{remark}

Theorem \ref{thm:Teich-has-limit} suggests that the behavior of the Teichm\"uller disks
in $T(\mathbb{D})$ is ``tame'' when compared to  the behavior of those in the finite dimensional Teichm\"uller space.
Namely,
when $S$ is a closed Riemann surface of genus at least two, Masur \cite{Mas} proved that the Teichm\"uller geodesics in $T(S)$ corresponding to holomorphic quadratic differentials with uniquely ergodic vertical foliations have a unique limit point on the Thurston boundary of $T(S)$. However, when the vertical foliation of a holomorphic quadratic differential on $S$ is not  uniquely ergodic the corresponding Teichm\"uller geodesic can have more than one limit point on the Thurston boundary of $T(S)$ (See \cite{CMW}, \cite{LLR}, \cite{Len}). 
As related topic,
for a closed Riemann surface $S$ of genus at least two and a holomorphic quadratic differential on $S$ whose vertical foliation is uniquely ergodic,
Jiang-Su \cite{JiangSu} and Alberge \cite{Alberge} proved that the corresponding horocyclic path has a unique limit point on the the Thurston boundary of $T(S)$ which is the projective class of the vertical foliation.
However,
to the authors knowledge,
the other cases (e.g. the convergence to $S^1$ outside any horoball based at $S^1$)
are still not understood for compact surfaces.

A geodesic current on $\mathbb{D}$ is a positive Radon measure on the space of geodesics of $\mathbb{D}$ (See \cite{Bon}, \cite{BonSar}, \cite{Sar1}. See also $\S$ 2). 
The Liouville map $\mathcal{L}$ maps the universal Teichm\"uller space $T(\mathbb{D})$ into the space of geodesic currents by taking the pull-back of the Liouville measure $L$ on the space of geodesics of $\mathbb{D}$ (See \cite{Bon}, \cite{BonSar}, \cite{Sar1}. See \S 2). The universal Teichm\"uller space $T(\mathbb{D})$ is homeomorphic to its image $\mathcal{L}(T(\mathbb{D}))$ inside the space of geodesic currents (See \cite{Bon} and \cite{BonSar}). 

By definition, the Thurston boundary of the universal Teichm\"uller space $T(\mathbb{D})$  consists of  limit points in the space of projective geodesic currents of the projectivization of $\mathcal{L}(T(\mathbb{D}))$. Let $h_{\lambda}:S^1\to S^1$ be the quasisymmetric map fixing $1$, $i$ and $-1$ that represents the Teichm\"uller deformation $D_{\varphi}(\lambda )$. Since $\mathcal{L}(h_{\lambda})\to\infty$ in the space of geodesic currents as $\lambda\to e^{i\theta}$, it is natural to consider the rate of convergence to the infinity.

\begin{theorem}[Asymptotics of Teichm\"uller deformation]
\label{thm:Teich-asymptotics}
Let $h_{\lambda}:S^1\to S^1$ be the quasisymmetric map that represents 
$D_\varphi(\lambda)$ in $T(\mathbb{D})$. Then, as $\lambda\to e^{i\theta}$,
$$ 
\frac{1-|\lambda |}{2\pi}\mathcal{L}(h_{\lambda})\to\mu_{e^{-i\theta}\varphi}.
$$
\end{theorem}

\subsection*{Acknowledgement}
The authors thank the referee for his/her fruitful comments and careful reading.

\section{Thurston boundary of the universal Teichm\"uller space}
\label{sec:Thurston-bordification}
Denote by $\mathbb{D}=\{ z:|z|<1\}$ the unit disk equipped with the hyperbolic metric $\frac{2|dz|}{1-|z|^2}$. Each oriented hyperbolic geodesic is uniquely determined by an ordered pair of distinct endpoints, the initial point and  end point of the geodesic. 
Therefore, the space of all oriented (hyperbolic) geodesics of $\mathbb{D}$ is identified with $(S^1\times S^1)\setminus diag$, where $diag$ is the diagonal of $S^1\times S^1$
and $S^1=\partial \mathbb{D}$. 
The space of geodesics contains a unique (up to a positive multiple) positive Borel measure of full support which is invariant under the isometries of $\mathbb{D}$, called the {\it Liouville measure}. It is defined by
$$
L(A)=\int_A\frac{d\alpha d\beta}{|e^{i\alpha}-e^{i\beta}|^2}
$$
for any Borel set $A\subset (S^1\times S^1)-diag$.  For a {\it box of geodesics} $[a,b]\times [c,d]$, with $a,b,c,d\in S^1$ given in the counterclockwise order, the Liouville measure is given by (See Bonahon \cite{Bon})
$$
L([a,b]\times [c,d])=\log\frac{(c-a)(d-b)}{(d-a)(c-b)}.
$$

The universal Teichm\"uller space $T(\mathbb{D})$ consists of all quasisymmetric maps $h:S^1\to S^1$ that fix $1$, $i$ and $-1$. 
A {\it geodesic current} on $\mathbb{D}$ is a positive Radon measure on the space of geodesics $(S^1\times S^1)-diag$.  Let $\mathcal{G}(\mathbb{D})$ be the space of geodesic currents on $\mathbb{D}$.
Bonahon \cite{Bon} introduced an embedding of the Teichm\"uller space $T(S)$ into the space of geodesic currents equipped with the weak* topology via the Liouville map, where $S$ is a closed surface of genus at least two. Moreover, $T(\mathbb{D})$ and $T(X)$, for $X$ any Riemann surface, embeds into the space of geodesic currents when equipped with the uniform weak* topology (See \cite{BonSar}, \cite{Sar1}). The {\it Liouville map} 
$$
\mathcal{L}:T(\mathbb{D})\to \mathcal{G}(\mathbb{D})
$$ 
is defined by
$$
\mathcal{L}(h)=(h^{-1})_{*}(L)
$$
where $h:S^1\to S^1$ is quasisymmetric and $(h^{-1})_{*}(L)$ is the push-forward of the Liouville measure $L$ by $h^{-1}$ (i.e. the pull-back of $L$ by $h$).

By definition, the Thurston boundary of $T(\mathbb{D})$ is the set of all boundary points of the image of $\mathcal{L}(T(\mathbb{D}))$ in the projective geodesic currents $P\mathcal{G}(\mathbb{D})$ equipped with the quotient of the uniform weak* topology. It turns out that the Thurston boundary consists of all projective bounded measured laminations $PML_b(\mathbb{D})$, where the projective class of a geodesic current $[\beta]\in P\mathcal{G}(\mathbb{D})$ is said to be a \emph{projective measured lamination} if the support of $\beta$ consists of non-intersecting geodesics. A projective measured lamination $[\beta]$ is said to be \emph{bounded} if $\sup_I \beta(I)$ is bounded where $I$ runs all transverse geodesic arc of unit length (See \cite{BonSar}, \cite{Sar1}). 

Given an integrable holomorphic function $\varphi :\mathbb{D}\to \mathbb{C}$, the corresponding Teichm\"uller geodesic is given by the equivalence class of quasiconformal maps
$$
g_t(z)=x+\frac{1-t}{1+t}yi
$$
where $0\leq t< 1$ and $z=\int_{*}\sqrt{\varphi (\zeta )}d\zeta $ is the natural parameter for $\mathbb{D}$ defined by $\varphi$.
Let $k_t:S^1\to S^1$ be a quasisymmetric map representing the equivalence class of $g_t$. 
Notice that the maximal dilatation $K(g_t)$ of $g_t$ is equal to $\frac{1+t}{1-t}$
when $0\leq t< 1$.

Let $\mu_{\varphi}$ be measured lamination on $\mathbb{D}$ whose support is the closure of the set of (hyperbolic) geodesics which have the same endpoints on $S^1$ as vertical trajectories of $\varphi$, i.e the leaves of $\mu_{\varphi}$ are obtained by straightening the vertical trajectories of $\varphi$. Since almost all vertical trajectories have two distinct limit points on $S^1$ (See \cite{Strebel}), it follows that to almost every vertical trajectory there corresponds a unique hyperbolic geodesic in the support of  $\mu_{\varphi}$. 

We define the $\mu_{\varphi}$-measure of a box of geodesics $[a,b]\times [c,d]$ as follows. If no vertical  trajectories have one endpoint in $[a,b]$ and another endpoint in $[c,d]$ then $\mu_{\varphi}([a,b]\times [c,d])=0$. In general, consider the set of all vertical trajectories that have one endpoint in $[a,b]$ and another endpoint in $[c,d]$. Let $\{ J_k\}_k$ be at most countable collection of compact subarcs of the  horizontal arcs of $\varphi$ such that every vertical trajectory with one endpoint in $[a,b]$ and another endpoint in $[c,d]$ intersects exactly one $J_k$, and no other vertical trajectory of $\varphi$ intersects any $J_k$. We define
$$
\mu_{\varphi}([a,b]\times [c,d]):=\sum_{k=1}^{\infty}\int_{J_k}\frac{\left|{\rm Re}(\sqrt{\varphi (\zeta )}d\zeta )\right|}{l(\zeta)}=\sum_k\int_{J_k'}\frac{|dx|}{l(z)}
 $$
 where $J_k'$ is the image of $J_k$ under the canonical coordinates $z=x+ti$ corresponding to $\varphi$ and $l(\zeta )$ is the $\varphi$-length of the vertical trajectory through $\zeta$, and similar for $l(z)$.

Then (See \cite{HaSar1})
$$
\lim_{t\to 1} K(g_t)^{-1}\mathcal{L}(k_t)= \mu_{\varphi}
$$
in the weak* topology on the geodesic currents $\mathcal{G}$, where $\mu_{\varphi}$ is the above measured lamination of $\mathbb{D}$. This implies that $\mu_\varphi$ is bounded. In general, an example showed that the above convergence does not hold for the uniform weak* topology (See \cite{HaSar1}).

\section{Modulus of a curve family}

Let $\Gamma$ be a family of curves in $\mathbb{C}$. A metric $\rho (z)|dz|$, where $\rho (z)\geq 0$ and measurable, is said to be {\it allowable for} $\Gamma$ if for every $\gamma\in\Gamma$ we have
$$
\int_{\gamma}\rho (z)|dz|\geq 1.
$$

The {\it modulus} of a curve family $\Gamma$ is given by
$$
\mathrm{mod}(\Gamma )=\inf_{\rho }\iint_{\mathbb{C}}\rho^2(z)dxdy
$$
where the infimum is over all allowable metrics $\rho (z)|dz|$ for the curve family $\Gamma$. 

The following properties of the modulus of families of curves are standard:
\begin{itemize}
  \item[1.] If $\Gamma_1\subset\Gamma_2$ then $\mathrm{mod} (\Gamma_1)\leq
      \mathrm{mod} (\Gamma_2)$.
  \item[2.]  $\mathrm{mod} (\bigcup_{i=1}^{\infty} \Gamma_i) \leq
      \sum_{i=1}^{\infty}\mathrm{mod} (\Gamma_i)$. The equality holds if $\Gamma_i$ and $\Gamma_j$ ($i\ne j$) are contained in mutually disjoint domains.
  \item[3.] If every $\gamma_2\in\Gamma_2$ contains some $\gamma_1\in\Gamma_1$ as a subcurve then $\mathrm{mod} (\Gamma_1 )\geq \mathrm{mod} 
      (\Gamma_2)$.
\end{itemize}

\section{Liouville measure of boxes and modulus of curves}

In \cite{HaSar2}, Hakobyan and
the second author observed that the Liouville measure and the modulus of a curve family are asymptotic to each other.
 
 \begin{lemma}[See \cite{HaSar2}]
\label{lem:mod_liouville_measure} Let $(a,b,c,d)$ be a quadruple of points on
$\mathbb{S}^1$ in the counterclockwise order. Let $\Gamma_{[a,b]\times
[c,d]}$ consist of all differentiable curves $\gamma$ in $\mathbb{D}$ which
connect $[a,b]\subset \mathbb{S}^1$ with $[c,d]\subset \mathbb{S}^1$. Then
$$
\mathrm{mod}(\Gamma_{[a,b]\times [c,d]})-\frac{1}{\pi}{L}([a,b]\times [c,d])-\frac{2}{\pi}\log 4\to 0
$$
as $\mathrm{mod}(\Gamma_{[a,b]\times [c,d]})\to\infty$, where ${L}$ is the
Liouville measure.
\end{lemma}

If $\mathrm{mod} (h_t(\Gamma_{[a,b]\times [c,d]}))\to \infty$ as $t\to\infty$, then Lemma \ref{lem:mod_liouville_measure} implies
\begin{equation}
\label{eq:mod/liou}
\lim_{t\to\infty}\frac{\mathrm{mod}( h_t(\Gamma_{[a,b]\times [c,d]}))}{\mathcal{L}(h_t)([a,b]\times [c,d])}\to\frac{1}{\pi}.
\end{equation}

Moreover, $\mathrm{mod} (h_t(\Gamma_{[a,b]\times [c,d]}))$ is bounded if and only if $\mathcal{L}(h_t)([a,b]\times [c,d])$ is bounded. 

\section{Proof of the Theorems}

\subsection{Convergence of Liouville measures}

For $s+ti\in \mathbb{H}_{>0}=\{w\in \mathbb{C}\mid {\rm Re}(w)>0\}$,
we consider quasiconformal mappings $f_{s+ti}:\mathbb{D}\to \mathbb{D}$ 
with  
Beltrami coefficients
$$
\frac{1-(s+ti)}{1+(s+ti)}\frac{\overline{\varphi}}{|\varphi|}.
$$
In the natural parameter $z=x+yi=\int_{*}\sqrt{\varphi(\zeta )}d\zeta$,
the corresponding quasiconformal mappings are 
$$
f_{s+ti}(z)=\frac{1}{s}x-\frac{t}{s}y+yi
$$ 
for $s+ti\in \mathbb{H}_{>0}$. Define
$$
D^{*}_{\varphi}:\mathbb{H}_{>0}\to T(\mathbb{D});\ \ \ D_{\varphi}^{*}(s+ti)=\Big{[}\frac{1-(s+ti)}{1+(s+ti)}\frac{\overline{\varphi}}{|\varphi|}\Big{]}.
$$
 We will prove

\begin{theorem}
\label{thm:rephrase-main-theorem}
For any box of geodesics $[a,b]\times [c,d]$, 
$$
\frac{s}{s^2+t^2}\mathrm{mod}(f_{s+ti}(\Gamma_{[a,b]\times [c,d]}))\to \mu_{-\varphi}([a,b]\times [c,d])
$$
as $s+|t|\to \infty$, where $\Gamma_{[a,b]\times [c,d]}$ is the set of arcs connecting $[a,b]$ and $[c,d]$ in $\mathbb{D}$.
\end{theorem}

\noindent
{\bf Proof of Theorem \ref{thm:Teich-has-limit}.} Let $A:\mathbb{H}_{>0}\to\mathbb{D}$ be given by $A(z)=\frac{-z+1}{z+1}$. Note that $A(0)=1$, $A(1)=0$ and  $A(\infty )=-1$. Let $\lambda (s+ti):=A(s+ti)$ and note that $\lambda\to -1$ if and only if $s+|t|\to\infty$. Since $\mu_{\varphi}$ has no atoms (See \cite{HaSar1}) and $(D_{\varphi}\circ A)(s+ti)=D_{\varphi}^{*}(s+ti )$ we have that $D_{\varphi}(\lambda )\to [\mu_{-\varphi}]$ as $\lambda\to -1$ by Theorem \ref{thm:rephrase-main-theorem} and Lemma \ref{lem:mod_liouville_measure}. Moreover, $D_{\varphi}(-e^{i\theta}\lambda )=[-e^{i\theta}\lambda\frac{\bar{\varphi}}{|\varphi |}]=[\lambda\frac{-\overline{e^{-i\theta}\varphi}}{|-e^{-i\theta}\varphi |}]=D_{-e^{-i\theta}\varphi}(\lambda )\to [\mu_{e^{-i\theta}\varphi}]$ as $\lambda\to -1$ by Theorem \ref{thm:rephrase-main-theorem}. This gives the first statement of Theorem \ref{thm:Teich-has-limit} since $-e^{i\theta}\lambda\to e^{i\theta}$ as $\lambda\to -1$.

We define the extension of the Teichm\"uller disk $D_\varphi$ 
to the closed disk $\overline{\mathbb{D}}$
by setting
$$
D_\varphi(e^{i\theta})=[\mu_{e^{-i\theta}\varphi}]\quad (e^{i\theta}\in S^1).
$$
We proved above that if $\{\lambda_n\}_n\subset \mathbb{D}$
is a sequence  converging to $e^{i\theta_{0}}\in S^1$ then $D_\varphi(\lambda_n)\to [\mu_{e^{-i\theta_0}\varphi}]$
as $n\to \infty$
in the Thurston bordification.
To finish the proof of continuity, let $e^{i\theta_n}\to e^{i\theta_{\infty}}$ as $n\to\infty$. We need to prove that $[\mu_{e^{-i\theta_n}\varphi}]\to [\mu_{e^{-i\theta_{\infty}}\varphi}]$ and this  
 follows by the diagonal argument. The injectivity of the extension of $D_{\varphi}$ to the unit circle $S^1$ follows by \cite[Theorem 2]{HaSar1}.

\subsection{Proof of Theorem \ref{thm:rephrase-main-theorem}}
An elementary computation gives the Beltrami coefficient of $f_{s+ti}(z)$ to be
$$
\mu (z)=\mu_{s+ti}(z)=\frac{1-(s+ti)}{1+(s+ti)}.
$$
Since $s^2+t^2\to \infty$ as $s+|t|\to \infty$,
the maximal dilatation $K(f_{s+ti})$ of $f_{s+it}$ behaves
\begin{align}
K(f_{s+ti})
&=\frac{\sqrt{(1+s)^2+t^2}+\sqrt{(1-s)^2+t^2}}{\sqrt{(1+s)^2+t^2}-\sqrt{(1-s)^2+t^2}}
\label{eq:K-at-infinity}
\\
&=
\frac{(\sqrt{(1+s)^2+t^2}+\sqrt{(1-s)^2+t^2})^2}{4s}
\nonumber
\\
&=
\frac{s^2+t^2}{4s}\left(
\sqrt{1+\frac{2s+1}{s^2+t^2}}+\sqrt{1-\frac{2s-1}{s^2+t^2}}
\right)^2
=
\frac{s^2+t^2}{s}\left(
1+o(1)
\right)
\nonumber
\end{align}
as $s+|t|\to \infty$.

{\it Acknowledgement.} We include the proposition below for the convenience of the reader. The main ideas and techniques are already contained in \cite{HaSar1}. In particular, the idea of estimating the modulus of a curve family under vertical shrinking by a subfamily of curves which have horizontal variation at most $\delta >0$ is already used in \cite{HaSar1}, \cite{HaSar2}. Lemma \ref{lem:mod_liouville_measure} from Section 3 also appears in \cite{HaSar1}, \cite{HaSar2}.

\begin{proposition}
\label{prop:upper}
Under the above notation
$$
\limsup_{s+|t|\to\infty} \frac{1}{s+\frac{t^2}{s}}
\mathrm{mod}(f_{s+ti}(\Gamma_{[a,b]\times [c,d]}))\leq 
\mu_{-\varphi}([a,b]\times [c,d]).
$$
\end{proposition}

\begin{proof}
Let $B=[a,b]\times [c,d]\subset (S^1\times S^1)-diag$ be a fixed box of geodesics. Let $\nu =\left|{\rm Im}(\sqrt{\varphi (\zeta )}d\zeta)\right|=|dy|$ be the horizontal foliation of $\nu$ which is also the vertical foliation of $-\varphi$.  Note that $\nu$ is a measured foliation of $\mathbb{D}$ while $\mu_{-\varphi}$ is a measured geodesic lamination of $\mathbb{D}$. 

Let $\Gamma_B$ be the family of all Jordan curves that connect $[a,b]$ to $[c,d]$ inside $\mathbb{D}$. Let $\delta >0$ be fixed. Define $\Gamma^{\geq\delta}_B$ to be the family of all $\gamma\in \Gamma_B$ such that $\nu (\gamma )\geq\delta$, and define $\Gamma^{<\delta}_B$ to be all $\gamma\in \Gamma_B$ such that $\nu (\gamma )<\delta$ (See \cite{HaSar1}). Since $f_{s+ti}$ does not change the $y$-coordinate in the canonical parameter of $\varphi$, we have that
curves in $\Gamma_B^{\geq \delta}$ are mapped onto curves in $\Gamma_{f_{s+ti}(B)}^{\geq\delta}$ and curves in $\Gamma_B^{< \delta}$ are mapped onto curves in $\Gamma_{f_{s+ti}(B)}^{<\delta}$.  

Define $\rho_{\delta} (z')=\frac{1}{\delta}$ in the canonical parameter $z'=f_{s+ti}(z)=x'+y'i$ of the terminal quadratic differential on $f_{s+ti}(\mathbb{D})$ corresponding to $\varphi$ and the map $f_{s+ti}$. Since $l_{\rho}(\gamma ):=\int_{\gamma}\rho_{\delta} (z')|dz'|\geq\int_{\gamma}\frac{1}{\delta}dy'$, we have
$$
l_{\rho}(\gamma )\geq\frac{1}{\delta}\delta =1
$$
for all $\gamma\in f_{s+ti}(\Gamma^{\geq\delta}_B)=\Gamma_{f_{s+ti}(B)}^{\geq\delta}$. Thus $\rho_{\delta}$ is allowable for the family
$f_{s+ti}(\Gamma_B^{\geq\delta})=\Gamma_{f_{s+ti}(B)}^{\geq\delta}$ and by the definition of the modulus $\mathrm{mod}(f_{s+ti}(\Gamma_B^{\geq\delta}))\leq \frac{1}{\delta^2}\int_{\mathbb{D}}|\varphi (\zeta )|d\xi d\eta$. We obtain
\begin{equation}
\label{eq:limsup1}
\lim_{s+|t|\to\infty}\frac{1}{s+\frac{t^2}{s}}\mathrm{mod}(f_{s+ti}(\Gamma_B^{\geq\delta}))\leq\lim_{s+|t|\to\infty}\frac{1}{s+\frac{t^2}{s}}\frac{1}{\delta^2}\int_{\mathbb{D}}|\varphi (\zeta )|d\xi d\eta=0.
\end{equation}
From \eqref{eq:K-at-infinity}
and the quasi-invariance of the modulus we obtain
\begin{equation}
\label{eq:limsup}
\lim_{s+|t|\to\infty}\frac{1}{s+\frac{t^2}{s}}\mathrm{mod}(f_{s+ti}(\Gamma_B^{<\delta}))\leq \lim_{s+|t|\to\infty}\frac{1}{s+\frac{t^2}{s}}K(f_{s+ti}) \mathrm{mod}(\Gamma_B^{<\delta})=\mathrm{mod}(\Gamma_B^{<\delta}).
\end{equation}
If all points on $S^1$ are on a finite $\varphi$-distance from a point in $\mathbb{D}$, then as $\delta\to 0$,
$\Gamma_B^{<\delta}$ converges to 
the set $|\nu_B|$ of horizontal trajectories of $\varphi$ that have one endpoint in $[a,b]$ and the other endpoint in $[c,d]$ (See \cite{HaSar1}).
Here,
a sequence of rectifiable curves $\gamma_n$ \emph{converges} to a curve $\gamma$ if there is a uniformly Lipschitz parametrizations of all curves by the same interval so that $\gamma_n$ converge uniformly to $\gamma$ as functions. 
A sequence $\Gamma_n$ of families of rectifiable curves has  \emph{limit} $\Gamma$ if $\Gamma$ consists of all curves $\gamma$ such that there is a sequence $\gamma_n\in\Gamma_n$ with $\gamma_n$ converges to $\gamma$ in the above sense.
By applying Keith's theorem (See \cite{Keith} and \cite{HaSar1}) we have
\begin{equation}
\label{eq:delta0}
\limsup_{\delta\to 0}\mathrm{mod}(\Gamma_B^{<\delta})\leq \mathrm{mod}(|\nu_B|)
\end{equation}
as $\delta\to 0$. Keith's theorem is stated for compact metric spaces and in the metric induced by an integrable holomorphic quadratic differentials some points of $S^1$ could be on infinite distance from the interior points.  Note that even when some points of $S^1$ are on infinite $\varphi$-distance the formula (\ref{eq:delta0}) holds (See \cite[proof of Theorem 1.4]{HaSar1}).

Therefore by (\ref{eq:limsup}) and (\ref{eq:limsup1}) we obtain
\begin{align*}
\limsup_{s+|t|\to\infty}\frac{1}{s+\frac{t^{2}}{s}}
\mathrm{mod}(f_{s+ti}(\Gamma_B))
&\leq 
\limsup_{s+|t|\to\infty}\frac{1}{s+\frac{t^{2}}{s}}\mathrm{mod}(f_{s+ti}(\Gamma_B^{<\delta}))
\\
&\qquad
+\limsup_{s+|t|\to\infty}\frac{1}{s+\frac{t^{2}}{s}}\mathrm{mod}(f_{s+ti}(\Gamma_B^{\geq\delta}))
\\
&=
\mathrm{mod}(\Gamma_B^{<\delta}).
\end{align*}
Since the left hand side of the above inequality does not depend on $\delta$ and the right hand side converges to $\mathrm{mod}|\nu^B|$ as $\delta\to 0$, we obtained 
$$
\limsup_{s+|t|\to\infty}\frac{1}{s+\frac{t^{2}}{s}}
\mathrm{mod}(f_{s+ti}(\Gamma_B))\leq \mathrm{mod}(|\nu^B|).
    $$
By the use of the Beurling's criteria 
(See \cite{Ahlfors}),
 we conclude that the metric $\frac{dy}{l(z)}$ in the canonical coordinates $z=x+yi$ of $\varphi$ is extremal for the family of curves $|\nu^B|$, where $l(z)$ is the $\varphi$-length of the horizontal trajectory through $z$. Consequently,
we get $\mathrm{mod}(|\nu^B|)=\mu_{-\varphi}(B)$ and the proof of Proposition \ref{prop:upper} is finished.
\end{proof}

We also need a converse inequality whose proof is  our main contribution. When the convergence is only along the Teichm\"uller geodesic this inequality follows essentially by Beurling's criteria (See \cite{HaSar1}) while the proof when the convergence is along an arbitrary sequence in $D_{\varphi}$ is more substantial. We first prove a special case of the converse inequality in the following lemma.

\begin{lemma}
\label{lem:rectangle}
Let $R=[0,a]\times [0,b]$ and $A\subset [0,b]$ be a measurable subset of a possibly positive Lebesgue measure $m(A)$. Let $\Gamma$ be the family of curves in $R$ that connects $\{0\}\times ([0,b]\setminus A)$ and $\{a\}\times ([0,b]\setminus A)$. Then
$$
\lim_{s+|t|\to \infty}\frac{1}{s+\frac{t^2}{s}}\mathrm{mod}(f_{s+ti}(\Gamma ))\geq\frac{b-2m(A)}{a}.
$$
\end{lemma}

\begin{figure}
\includegraphics[width=10cm, bb=0 0 589 281]{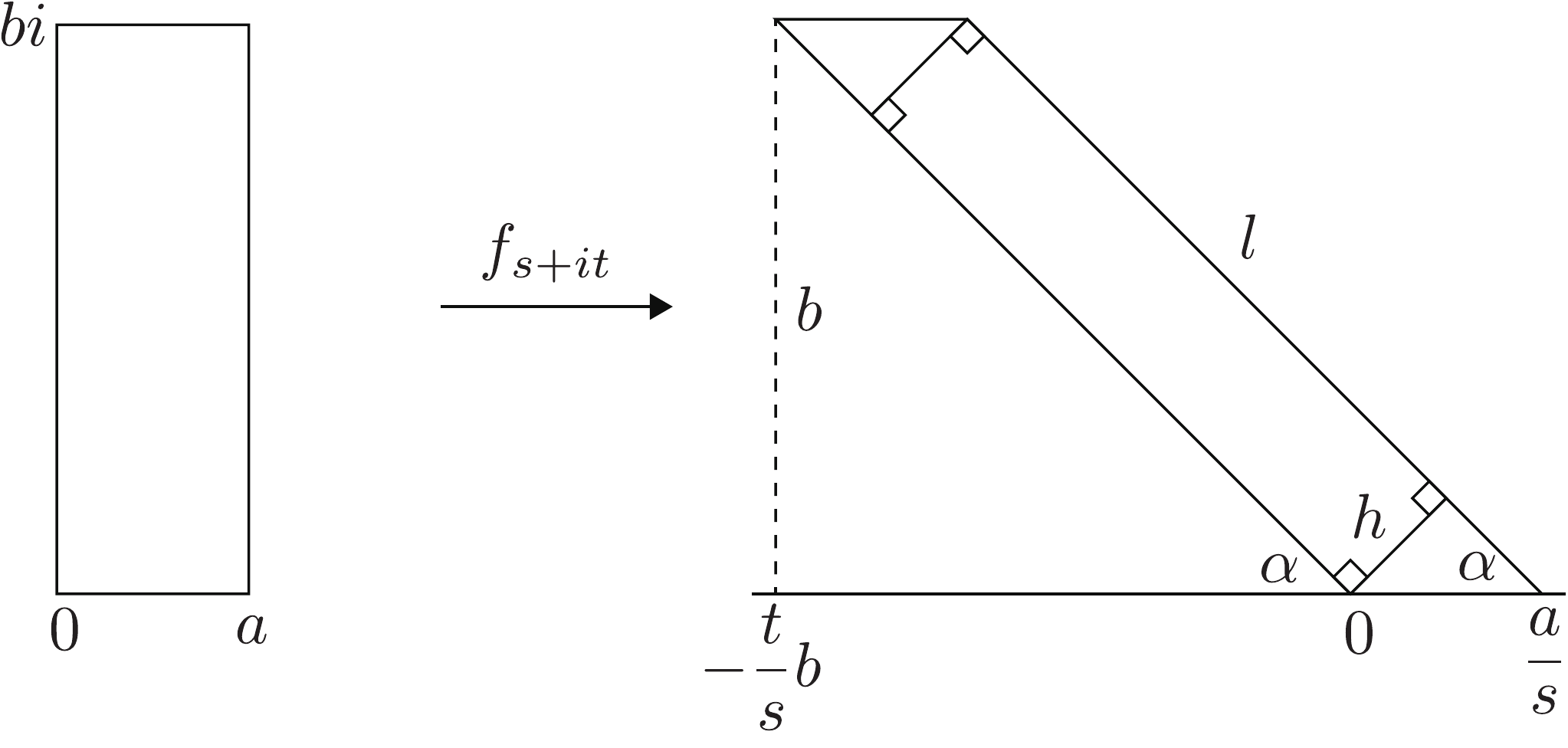}
\caption{The parallelogram $R_{s+ti}=f_{s+ti}(R)$.} 
\label{fig:slanted-rectangle}
\end{figure}

\begin{proof}
From the large right angled triangle in Figure \ref{fig:slanted-rectangle} we get
$$
\tan\alpha =\frac{b}{\frac{|t|}{s}b}=\frac{s}{|t|}.
$$
Let $h$ be the Euclidean length of the orthogonal arcs to the two slanted sides of $f_{s+ti}(R)$. Then the small right-angled triangle gives
$$
\sin\alpha =\frac{h}{\frac{1}{s}a}.
$$
The above two equalities give
$$
h=\frac{a}{s}
\frac{\frac{s}{|t|}}{\sqrt{1+\frac{s^2}{t^2}}}=\frac{a}{\sqrt{s^2+t^2}}.
$$

Let $l$ be the length of the subarc of the slanted side of $f_{s+ti}(R)$ that consists of endpoints of the arcs of length $h$ orthogonal to both slanted sides. Then we have
$$
l=\sqrt{b^2+\frac{t^2}{s^2}b^2}-\cos\alpha \frac{a}{s}=\frac{b}{s}\sqrt{s^2+t^2}-\frac{1}{\sqrt{1+\frac{s^2}{t^2}}}\frac{a}{s}
$$
which gives
$$
l=\frac{b(s+\frac{t^2}{s})-a\frac{|t|}{s}}{\sqrt{s^2+t^2}}.
$$

The family of curves $f_{s+ti}(\Gamma )$ contains a subfamily $\Gamma_{s+ti}'$ of all line segments that are orthogonal at both ends to the slanted boundary sides of $f_{s+ti}(R)$ except the line segments that have at least one endpoint in $f_{s+ti}((\{ 0\}\times A)\cup (\{ a\}\times A))$. The total Lebesgue measure of the set of endpoints of $\Gamma_{s+ti}'$ is at least $l-2m(A)\sqrt{1+\frac{t^2}{s^2}}$. This gives
$$
\mathrm{mod} (f_{s+ti}(\Gamma ))\geq \mathrm{mod}(\Gamma_{s+ti}')\geq \frac{l-2m(A)\sqrt{1+\frac{t^2}{s^2}}}{h}.
$$
Using the above estimates we obtain
\begin{align*}
\mathrm{mod} (f_{s+ti}(\Gamma ))
&\geq \Big{[}\frac{b(s+\frac{t^2}{s})-a\frac{|t|}{s}}{\sqrt{s^2+t^2}}-2m(A)\sqrt{1+\frac{t^2}{s^2}}\Big{]}:\frac{a}{\sqrt{s^2+t^2}}\\
&=\frac{b}{a}(s+\frac{t^2}{s})-\frac{|t|}{s}-\frac{2m(A)}{a}\frac{s^2+t^2}{s}
\end{align*}
and
$$
\liminf_{s+|t|\to\infty} \frac{\mathrm{mod}( f_{s+ti}(\Gamma ))}{s+\frac{t^2}{s}} \geq \liminf_{s+|t|\to\infty} \Big{[}
\frac{b}{a} -\frac{|t|}{s(s+\frac{t^2}{s})}-\frac{2m(A)}{a}\Big{]}.
$$

Note that $\limsup_{s+|t|\to\infty} \frac{|t|}{s(s+\frac{t^2}{s})}=\limsup_{s+|t|\to\infty}\frac{|t|}{s^2+t^2}=0$ and
 we obtain
 $$
\liminf_{s+|t|\to\infty} \frac{\mathrm{mod} (f_{s+ti}(\Gamma ))}{s+\frac{t^2}{s}} \geq \frac{b-2m(A)}{a}
$$
\end{proof}

We consider the following situation. Let $D=\{ z=x+yi| y\in [0,s_0], h_2(y)<x<h_1(y)\}$ where $h_1$ is lower semicontinuous and $h_2$ upper semicontinuous function such that $h_1(y)>c_1>0$ and $h_2(y)<c_2<0$. Then $D$ is a domain in $\mathbb{C}$ and we further assume that $D$ has finite Lebesgue area. Denote by $\Gamma_D$ the family of curves that connects the graphs of $h_1$ and $h_2$ inside $D$.
Let $R=[a,b]\times [0,s_0]$ be a rectangle that contains the graphs of $h_1$ and $h_2$ over $[0,s_0]-A$, where $A\subset [0,s_0]$ is Lebesgue measurable and $2m(A)<s_0$.
Then we prove

\begin{lemma}
\label{lem:one_rectangle_approx}
Under the above notation,
$$
\liminf_{s+|t|\to\infty}\frac{1}{s+\frac{t^2}{s}}\mathrm{mod}(f_{s+ti}(\Gamma_D))\geq\frac{s_0-2m(A)}{b-a}.
$$
\end{lemma}

\begin{proof}
Let $D_{s+ti}:=f_{s+ti}(D)$ and $R_{s+ti}:=f_{s+ti}(R)$. Let $\Gamma_{s+ti}^{\perp}$ be the family of curves that consists of orthogonal segments to the slanted sides of the parallelogram $R_{s+ti}$ such that the endpoints of each $\gamma\in\Gamma_{s+ti}^{\perp}$ do not belong to $(\{ a\}\times A)\cup(\{b\}\times A)$ (cf. Figure \ref{fig:slanted-rectangle}).  We claim that each $\gamma\in \Gamma_{s+ti}^{\perp}$ contains a subcurve in $f_{s+ti}(\Gamma_D)$.

Indeed, since the graphs of $h_1$ and $h_2$ over $[0,s_0]-A$ are in $R$ it follows that both endpoints of $\gamma$ are outside $D_{s+ti}$. On the other hand $\gamma$ intersects $D_{s+ti}$ and does not intersect the real axis or the line parallel to the real axis through point $s_0i$. This implies that $\gamma$ intersects the images under $f_{s+ti}$ of both graphs of $h_1$ and $h_2$ over $[0,s_0]$. Therefore there is a subsegment $\gamma'$ of $\gamma$ that is in $f_{s+ti}(\Gamma_D)$. 

By the monotonicity of the module (See \S 3, property 3) we have 
$$
\mathrm{mod}(f_{s+ti}(\Gamma_D))\geq\mathrm{mod}(\Gamma_{s+ti}^{\perp}).
$$
The family of curves $\Gamma_{s+ti}^{\perp}$ is used in the proof of Lemma \ref{lem:rectangle} to obtain the lower bound. Therefore the lower bound in Lemma \ref{lem:rectangle} implies 
$$
\liminf_{s+|t|\to\infty}\frac{1}{s+\frac{t^2}{s}}\mathrm{mod}(f_{s+ti}(\Gamma_D))\geq 
\frac{s_0-2m(A)}{b-a}.
$$
\end{proof}
 
 We are ready to prove the lower bound for the general case.

\begin{proposition} 
\label{prop:lower}
Under the above notation we have $$
\liminf_{s+|t|\to\infty}\frac{1}{s+\frac{t^2}{s}}\mathrm{mod}(f_{s+ti}(\Gamma_B))\geq \mu_{-\varphi}([a,b]\times [c,d]),
$$
for all boxes of geodesics $[a,b]\times [c,d]\subset (S^1\times S^1)-diag$.
\end{proposition}

\begin{proof}
Consider the set of all horizontal trajectories of $\varphi$ that connect $[a,b]$ to $[c,d]$. They are divided into at most countably many horizontal strips $\{ S_{\omega}\}_{\omega =1}^{\infty}$ such that $S_{\omega}$ and $S_{\omega'}$ can have in common at most one of their boundary horizontal trajectories for $\omega\neq \omega'$. Let $\beta_{\omega}$ be a segment of the vertical trajectory that defines $S_\omega$ (See Strebel \cite{Strebel}). Denote by $\Gamma_\omega$ the family of curves in the strip $S_{\omega}$ that connects the vertical sides of $S_{\omega}$. Then $\Gamma_\omega\subset \Gamma_B$ and by monotonicity and  additivity
$$
\mathrm{mod}(f_{s+ti}(\Gamma_B))\geq\sum_{\omega}\mathrm{mod}(f_{s+ti}(\Gamma_{\omega})).
$$

We fix $\omega$ and estimate $\mathrm{mod}(f_{s+ti}(\Gamma_{\omega}))
$.
Let $D$ be the image of $S_\omega$ in the natural parameter such that $\beta_\omega$ is mapped onto the interval $[0,s_0]$ of the $y$-axis. 
Then there exist $h_1,h_2:[0,s_0]\to\mathbb{R}\cup\{-\infty ,\infty\}$ such that 
$$
D=\{x+yi\in \mathbb{C}\mid  0<y<s_0, h_2(y)<x<h_1(y)\},
$$
$h_1(y)>0>h_2(y)$ for all $y\in [0,s_0]$ and that $\liminf_{y\to y_0}h_1(y)\geq h_1(y_0)$ and $\limsup_{y\to y_0}h_2(y)\leq h_2(y_0)$ for all $y_0\in [0,s_0]$. The function $h_1$ is lower semicontinuous and the function $h_2$ is upper semicontinuous. Thus $h_1$ has a minimum $c_1>0$ and $h_2$ has a maximum $c_2<0$ on $[0,s_0]$.

 In particular, both $h_1:[0,s_0]\to [c_1,\infty ]$ and $h_2:[0,s_0]\to [-\infty ,c_2]$ are  Lebesgue measurable and thus so are $1/h_1:[0,s_0]\to [0,1/c_1]$ and $1/h_2:[0,s_0]\to [1/c_2,0]$.
 By a corollary to Lusin's theorem applied to $1/h_1$ and $1/h_2$ (See Rudin \cite[page 56]{Rudin}),  there exist sequences  of continuous functions $1/g_n:[0,s_0]\to [0,1/c_1]$ and
 $1/f_n:[0,s_0]\to [1/c_2,0]$ such that $1/g_n(y)\to 1/h_1(y)$ and $1/f_n(y)\to 1/h_2(y)$ as $n\to\infty$ for a.a. $y\in [0,s_0]$. Since $g_n(y)-f_n(y)\geq c_1-c_2$ (again by Lusin's theorem), Lebesgue Dominated Convergence Theorem implies that
 \begin{equation}
 \label{eqn:convergence_funct}
 \int_B\frac{dy}{g_n(y)-f_n(y)}\to\int_B\frac{dy}{h_1(y)-h_2(y)}
 \end{equation}
 for any Lebesgue measurable set $B\subset [0,s_0]$ as $n\to\infty$. 
 
We fix $n$ and divide the interval $[0,s_0]$ into $p$ subintervals $I_k=[\frac{k-1}{p}s_0,\frac{k}{p}s_0)$ for $k=1,2,\dots ,p$. Define $s_p^1(y)=\max_{I_k}g_n$ and $s_p^2(y)=\min_{I_k}f_n$ for $y\in I_k$. The two step functions satisfy $s_p^1\geq g_n$ and $s_p^2\leq f_n$ on $[0,s_0]$. Since $1/(g_n-f_n)$ is bounded and continuous on $[0,s_0]$, it follows that
$$
\int_{B}\frac{1}{s_p^1(y)-s_p^2(y)}dy\to \int_{B}\frac{1}{g_n(y)-f_n(y)}dy
$$
for any Lebesgue measurable $B\subset [0,s_0]$ as $p\to\infty$. 

Denote by $A_{n}\subset [0,s_0]$ the set of all $y$ such that either $g_n(y)\neq h_1(y)$ or $f_n(y)\neq h_2(y)$. The Lebesgue measure $m(A_{n})$ is going to zero as $n\to\infty$ by Lusin's theorem. 
Let $D_p$ denote the domain between the graphs of $s_p^1$ and $s_p^2$. 
Since $s_p^j$ for $j=1,2$ each have $p$ steps we conclude that $D_p$ is the union of $p$ rectangles $D_{p,k}$ for $k=1,2,\ldots ,p$ whose vertical sides are the steps of $s_p^j$ for $j=1,2$. Let $A_{n,k}=A_n\cap I_k$.

We consider the curve family $(\Gamma^{p,k}_{s+ti})^{\perp}$ that consists of all Euclidean segments orthogonal at both endpoints to the slanted sides of the parallelogram $f_{s+ti}(D_{p,k})$ such that the endpoints do not belong $f_{s+ti}[\{ (s_p^1(y),y)|y\in A_{n,k}\}\cup \{ (s_p^2(y),y)|y\in A_{n,k}\} ]$. 

By Lemma \ref{lem:one_rectangle_approx} we get that
$$
\liminf_{s+|t|\to\infty}\frac{1}{s+\frac{t^2}{s}}\mathrm{mod}((\Gamma^{p,k}_{s+ti})^{\perp})\geq
\frac{s_0-2m(A_{n,k})}{s_p^1(y_{k})-s_p^2(y_{k})}
$$
where $y_k\in I_k$ is arbitrary. Since $(\Gamma^{p,k}_{s+ti})^{\perp}$ for $k=1,2,\ldots ,p$ are pairwise disjoint and each curve in $(\Gamma^{p,k}_{s+ti})^{\perp}$ contains a curve in $f_{s+ti}(\Gamma_\omega)$,  we get
$$
\liminf_{s+|t|\to\infty}\frac{1}{s+\frac{t^2}{s}}\mathrm{mod}(f_{s+ti}(\Gamma_\omega))\geq\sum_{k=1}^p\frac{m(I_k)}{s_p^1(y_{k})-s_p^2(y_{k})}-2\sum_{k=1}^p\frac{m(A_{n,k})}{s_p^1(y_{k})-s_p^2(y_{k})}
$$
which gives
$$
\liminf_{s+|t|\to\infty}\frac{1}{s+\frac{t^2}{s}}\mathrm{mod}(f_{s+ti}(\Gamma_\omega))\geq\int_{[0,s_0]}\frac{dy}{s_p^1(y)-s_p^2(y)}-2\int_{A_n}\frac{dy}{s_p^1(y)-s_p^2(y)}.
$$
By letting $p\to\infty$ in the above inequality we obtain
$$
\liminf_{s+|t|\to\infty}\frac{1}{s+\frac{t^2}{s}}\mathrm{mod}(f_{s+ti}(\Gamma_\omega))\geq\int_{[0,s_0]}\frac{dy}{g_n(y)-f_n(y)}-2\int_{A_n}\frac{dy}{g_n(y)-f_n(y)}.
$$
Since $\int_{A_n}\frac{dy}{g_n(y)-f_n(y)}\geq \frac{m(A_n)}{c_2-c_1}\to 0$ as $n\to\infty$, by letting $n\to\infty$ the above inequality gives
$$
\liminf_{s+|t|\to\infty}\frac{1}{s+\frac{t^2}{s}}\mathrm{mod}(f_{s+ti}(\Gamma_\omega))\geq \int_{[0,s_0]}\frac{dy}{h_1(y)-h_2(y)}=\mu_{-\varphi}(\beta_{\omega}).
$$
The proposition follows by summing the above inequality over all $\omega$.
\end{proof}

Let $h_{\lambda}:S^1\to S^1$ be the quasisymmetric map representing $D_{\varphi}(\lambda )$.
Propositions \ref{prop:upper} and \ref{prop:lower} imply the projective weak* convergence of $\mathcal{L}(h_{\lambda})$ to $\mu_{e^{-i\theta}\varphi}$ as $\lambda\to e^{i\theta}$ and this finishes the proof of Theorem \ref{thm:rephrase-main-theorem}.

\subsection{Proof of Theorem \ref{thm:Teich-asymptotics}} If $B=[a,b]\times [c,d]$ is a box of geodesics, in the previous section we established that $\lim_{s+|t|\to\infty}\frac{1}{s+\frac{t^2}{s}}\mathrm{mod}(f_{s+ti}(\Gamma_B))=\mu_{-\varphi}(B)$. Let $A(z)=\frac{-z+1}{z+1}$ and $\lambda =A(s+ti)$. 
Note that $s+ti =\frac{-\lambda+1}{\lambda +1}$ and $s=\frac{1-|\lambda |^2}{|1-\lambda |^2}$.
Then $s+|t|\to\infty$ if and only if $\lambda\to -1$ and we have 
$$
\lim_{s+|t|\to\infty}\Big{[}\frac{s^2+t^2}{s}\Big{]}/\Big{[}\frac{2}{1-|\lambda |}\Big{]}=1. 
$$
Then Lemma \ref{lem:mod_liouville_measure} implies $\lim_{\lambda\to -1}\frac{1-|\lambda |}{2\pi}\mathcal{L}(h_{\lambda})=\mu_{-\varphi}$.

Let $B_{\theta}(z)=-e^{-i\theta}z$. Then we have
$$
D_{\varphi}(\lambda )=\left[\lambda\frac{\bar{\varphi}}{|\varphi |}\right]=\left[B_{\theta}(\lambda )\frac{-\overline{e^{-i\theta}\varphi}}{|\varphi |}\right].
$$
Note that $B_{\theta}(\lambda )\to -1$ if and only if $\lambda\to e^{i\theta}$. Then the above limit implies that
$\lim_{\lambda\to e^{i\theta}}\frac{1-|\lambda |}{2\pi}\mathcal{L}(h_{\lambda})=\mu_{e^{-i\theta}\varphi}$ for each $e^{i\theta}\in S^1$ and the proof is completed.

\end{document}